\documentclass[a4paper]{article}

\usepackage{
amsmath,
amsthm,
amscd,
amssymb,
stmaryrd,
wasysym
}
\usepackage{comment}
\usepackage{tikz}
\usetikzlibrary{positioning,arrows,calc}
\usepackage{xspace}

\usepackage{graphicx}

\usepackage{url}
\theoremstyle{plain}
\newtheorem{lemma}{Lemma}
\newtheorem{definition}{Definition}
\newtheorem{corollary}{Corollary}
\newtheorem{proposition}{Proposition}
\newtheorem{theorem}{Theorem}

\newtheorem{question}{Question}

\newtheoremstyle{derp}% <name>
{3pt}% <Space above>
{3pt}% <Space below>
{}% <Body font>
{}% <Indent amount>
{\upshape}% <Theorem head font>
{:}% <Punctuation after theorem head>
{.5em}% <Space after theorem headi>
{}% <Theorem head spec (can be left empty, meaning `normal')>
\theoremstyle{derp}
\newtheorem{example}{Example}

\newcommand{\Z}{\mathbb{Z}}

\newcommand{\N}{\mathbb{N}}

\newcommand\xqed[1]{%
  \leavevmode\unskip\penalty9999 \hbox{}\nobreak\hfill
  \quad\hbox{#1}}
\newcommand\qee{\xqed{$\fullmoon$}}

\newcommand{\Aut}{\mathrm{Aut}}

\newcommand{\Sym}{\mathrm{Sym}}
\newcommand{\Alt}{\mathrm{Alt}}

\title{Wreath products in the automorphism group of a full shift}

\author{
Ville Salo \\
vosalo@utu.fi
}

\begin{document}
\maketitle

\begin{abstract}
We prove that if a subgroup $H$ of the automorphism group $\Aut(\Sigma^\Z)$ of a non-trivial full shift acts on points of finite support (= points bi-asymptotic to a fixed point) with a free orbit, then for every finitely-generated abelian group $A$, the abstract group $A \wr H$ also embeds in $\Aut(\Sigma^\Z)$. The groups admitting an action with such a free orbit include $A \wr \Z$ for $A$ a finite abelian group, and finitely-generated free groups. The class of such groups is also closed under commensurability and direct products. We obtain for example that $\Z \wr \Z$, $\Z_2 \wr (\Z_2 \wr \Z)$ and $\Z \wr (\Z_2 \wr \Z)$ embed in $\Aut(\Sigma^\Z)$. The group $\Z \wr \Z$ is the first example of a finitely-generated torsion-free subgroup of $\Aut(\Sigma^\Z)$ with infinite cohomological dimension, answering an implicit question of Kim and Roush and an explicit question of the author. We also explore a simpler variant of the construction that gives embeddings of certain Neumann groups, as well as some near-misses to higher iterated wreath products.
\end{abstract}

\section{Introduction}

%A \emph{subshift} is a closed shift-invariant subset of $\Sigma^\Z$ for a finite alphabet $\Sigma$, where the shift is $\sigma(x)_i = x_{i+1}$. An important example is the \emph{full shift}, namely $\Sigma^\Z$ itself. %An \emph{SFT} is a subshift consisting of exactly the configurations whose orbits miss some clopen set $C$. A \emph{sofic shift} is a subshift that is the image of a SFT under a shift-commuting continuous function. The \emph{automorphism group} $\Aut(X)$ of a subshift $X$ is the group of shift-commuting homeomorphisms $f : X \to X$.

The \emph{full shift} is $\Sigma^\Z$, where $\Sigma$ is a finite alphabet, seen as a dynamical system under the $\Z$-action of the shift map $\sigma$, defined by $\sigma(x)_i = x_{i+1}$. Its automorphism group $\Aut(\Sigma^\Z)$ consists of $\sigma$-commuting self-homeomorphisms of $\Sigma^\Z$, also known as reversible cellular automata.

This paper continues the study of the family
\[ \mathcal{G} = \{ G \mbox{ group} \;|\; G \hookrightarrow \Aut(\Sigma^\Z) \} \]
of groups that abstractly embed into the automorphism group of $\Z$-full shift, which we sometimes call \emph{groups of cellular automata}, the study of which begun in \cite{He69,BoLiRu88}. The main ``upper bounds'' known on $\mathcal{G}$ are that groups in it are countable \cite{He69}, residually finite \cite{BoLiRu88}, and the word problems of f.g.\ groups in $\mathcal{G}$ are in co-NP \cite{Sa20e,KiRo90,BoLiRu88}. %\footnote{As was pointed out to me by Mike Boyle, the word problems are in co-NTIME$(n^d)$ for a fixed $d$, which is indeed a stronger restriction than co-NP.}

There is much more work on the constructive side, i.e.\ many complicated behaviors have been exhibited in f.g.\ groups in $\mathcal{G}$. The Tits' alternative can fail \cite{Sa17d}, the torsion problem can be undecidable \cite{BaKaSa23,Sa22b}, there can be distorted infinite cyclic subgroups \cite{CaSa22}, and the conjugacy problem is in a sense undecidable in all large enough groups in $\mathcal{G}$ \cite{Sa20e}. We also know that $\mathcal{G}$ is closed under commensurability \cite{KiRo90} and countable graph products \cite{Sa18d,Sa20c}.

Less is known in the ``tame'' end of the spectrum of groups. In particular, very little is known about solvable groups in $\mathcal{G}$. Indeed, already abelian groups are a mystery: all finitely-generated abelian groups are embeddable due to closure properties of $\mathcal{G}$, but we already do not know whether the additive group of dyadic rationals $(\Z[1/2], +)$ is in $\mathcal{G}$. Residual finiteness does not prevent this, and in general it is a major open problem whether $\Aut(\Sigma^\Z)$ can have any elements of infinite order with roots of infinitely many orders \cite{Bo08}. For general subshifts, such behavior is exhibited in \cite{BoLiRu88,Sa14d}. 

Next, one may consider nilpotent groups. Now, whether there are \emph{any} finitely-generated nilpotent groups in $\mathcal{G}$ (which are not virtually abelian) is a major open problem. The issue is that such groups always contain a copy of the three-dimensional integer Heisenberg group with presentation
\[ \langle a, b \;|\; [a, [a, b]], [b, [a, b]] \rangle, \]
and this group has a distorted $\Z$-subgroup (namely $\langle [a, b] \rangle$). We only know a single example of a distorted $\Z$-subgroup in the automorphism group of a $\Z$-subshift \cite{CaSa22} and have no reason to believe it is the center of a Heisenberg subgroup. The question of whether the Heisenberg group is in $\mathcal{G}$ is essentially due to Kim and Roush \cite{KiRo90}.

Next, one may consider metabelian groups. Again we are in trouble already with the most basic finitely-generated examples, as the Baumslag-Solitar group $\Z[1/2] \rtimes \Z$ contains a copy of $\Z[1/2]$. It also has distortion elements, and here the situation is worse than with the Heisenberg group, as the distortion function is exponential (much faster than the rate demonstrated for the example in \cite{CaSa22}).

In light of these difficult problems a more promising class to consider are wreath products. To guide us to which types of wreath products we should consider, we should consider the known restrictions: countability, residual finiteness, and the complexity of the word problem.

Countability gives no interesting restrictions. The fact $\Aut(\Sigma^\Z)$ is residually finite, on the other hand, gives nontrivial restrictions. We recall Gruenberg's theorem \cite{Gr57}: The group $G = K \wr H$ is residually finite if and only if $K, H$ are residually finite and $H$ finite or $K$ abelian. For $H$ finite the wreath product embedding problem was completely solved by Kim and Roush \cite{KiRo90}, namely $\Aut(A^\Z) \wr H \leq \Aut(A^\Z)$ for all such $H$. After taking this into account, the co-NPness of the word problem does not pose further restrictions, in the sense that if $A$ is f.g.\ abelian and $B$ has co-NP word problem, then $A \wr B$ has co-NP word problem. (See Lemma~\ref{lem:WreathNP} for a proof and further discussion.)

Thus a natural question is:

\begin{question}
Is $\mathcal{G}$ closed under wreath products with abelian base? More precisely, if $K$ is an abelian group in $\mathcal{G}$ and $H \in \mathcal{G}$, do we always have $K \wr H \in \mathcal{G}$?
\end{question}

Of course the answer is positive if and only if it is positive for the choice $H = \Aut(\Sigma^\Z)$. Before the present paper, the state-of-the-art with solvable groups obtained from wreath product constructions is \cite{Sa20c}, where we proved in particular the following theorem:

\begin{theorem}[\cite{Sa20c}]
For $n \geq 1$, the groups $\Z_2 \wr \Z^n$ are in $\mathcal{G}$.
\end{theorem}

(Here $\Z_2$ is the group with two elements.) More generally, we constructed the groups $A \wr H$ for some groups $H$, including free abelian groups and free non-abelian groups, and $A$ a finite abelian group. The cases $\Z \wr \Z$ and $\Z_2 \wr (\Z_2 \wr \Z)$ were left open in \cite{Sa20c}. Here, we solve these problems and slightly more:

\begin{theorem}
\label{thm:MainUgly}
If $A$ is a finitely-generated abelian group, $B$ is a finite abelian group, $C$ is a finitely-generated free group, and $n \in \N$, then $A \wr ((B \wr \Z)^n \times C^n)$ is in $\mathcal{G}$.
\end{theorem}

The statement above is a little cluttered; it is simply a practical way to summarize (more or less) all of the (non-permutational) wreath products we were able to obtain from the main construction. % (without finding new pointy actions).

\begin{corollary}
The groups $\Z \wr \Z^n$,
%\item
$\Z_2 \wr (\Z_2 \wr \Z)$, and
%\item
$\Z \wr (\Z_2 \wr \Z)$
are in $\mathcal{G}$.
\end{corollary}

Kim and Roush wrote in 1990 \cite{KiRo90} that all known finitely generated torsion free groups in $\mathcal{G}$ known at the time had a subgroup of finite index which is ``locally nilpotent and of finite cohomological dimension''. Roush clarified in private communication \cite{Ro18} that ``locally nilpotent'' should say ``residually nilpotent''. %; indeed, examples of finitely-generated torsion-free groups in $\mathcal{G}$ which are not locally nilpotent in the standard sense (which in this situation would simply mean nilpotent) were already constructed in \cite{BoLiRu88}. The construction of subgroups that are not residual nilpotent .

Despite us nowadays knowing many examples of interesting groups in $\mathcal{G}$, to our knowledge the only known interesting finitely-generated torsion-free groups exhibited before the present paper are still precisely the graph groups from \cite{KiRo90,BoLiRu88}, which are always residually nilpotent \cite[Theorem~2.3]{DuKr92}, and of finite cohomological dimension \cite{KiKo14}. The group $\Z \wr \Z$ of course has infinite cohomological dimension, as it contains $\Z^d$ for all $d$, providing the first solution to (a part of) the problem implicit in Kim and Roush's statement (of finding groups without the properties listed). %, so our theorem provides the first example to which the observation of Kim and Roush no longer applies.

All torsion-free groups provided by the theorem above are residually nilpotent (see Lemma~\ref{lem:AllTFN}), so the other half of the observation of Kim and Roush remains valid. On the side we note that there are certainly many non-(residually nilpotent) groups \emph{with} torsion in $\mathcal{G}$, simply because $\mathcal{G}$ contains all finite groups \cite{He69}; it is easy to also find examples that are not even virtually residually nilpotent, for example there are f.g.\ groups in $\mathcal{G}$ containing copies of every finite group (this follows from the main result of \cite{Sa22b}, or can be proved directly).

%for example a group that contains copies of all finite groups cannot be residually nilpotent, and the f.g.-universal groups from \cite{Salo} contain copies of all finite groups, and much more. It is also possible that our main technical result Theorem~\ref{thm:MainTechnical} can be used to produce examples, but at the moment we do not know many examples of interesting pointy actions by .

The methods of the present paper still do not get very far into iterated wreath products:

%The following question in a sense collects the ``frontier'' of iterated wreath products that do not seem to be within reach with the method of this paper.

% Our method fails for $\Z_2 \wr (\Z_2 \wr \Z^2)$ because we have not been able to find a pointy action of $\Z_2 \wr \Z^2$. It fails for $\Z_2 \wr (\Z \wr \Z)$ and $\Z_2 \wr (\Z_2 \wr (\Z_2 \wr \Z))$ for a similar reason.

\begin{question}
\label{q:AreThe}
Which of $\Z_2 \wr (\Z_2 \wr \Z^2)$, $\Z_2 \wr (\Z \wr \Z)$, $\Z_2 \wr (\Z_2 \wr (\Z_2 \wr \Z))$ are in $\mathcal{G}$?
%\begin{itemize}
%\item $\Z_2 \wr (\Z_2 \wr \Z^2)$,
%\item $\Z_2 \wr (\Z \wr \Z)$,
%\item $\Z_2 \wr (\Z_2 \wr (\Z_2 \wr \Z))$?
%\end{itemize}
\end{question}

% We give a conditional answer to this question: if the polynomial hierarchy does not collapse, then there is a finitely-generated group $H$ which embeds in $\Aut(\Sigma^\Z)$ such that $\Z_2 \wr H$ does not embed in $\Aut(\Sigma^\Z)$.

In \cite{Sa20c}, we obtained embeddings of $\Z_2 \wr \Z^n$ also in the one-sided setting (with varying alphabet sizes). The construction of the present paper does not seem to admit such a variant.

\begin{question}
Are $\Z \wr \Z$, $\Z_2 \wr (\Z_2 \wr \Z)$ and $\Z \wr (\Z_2 \wr \Z)$ groups of one-sided cellular automata, i.e.\ do they embed in the automorphism group of a full $\N$-shift?
\end{question}

The main new technical idea of this paper is a modification to the standard conveyor belt trick by allowing ``floating boundaries'' around the support of a configuration: the $0$s around the support are replaced with $>$s on the left, and $<$s on the right.

This contruction is explained in the proof of Theorem~\ref{thm:GinG} in general, and Lemma~\ref{lem:PointyWreath} shows how to get wreath products out of it. In Example~\ref{ex:ZZ2Z} we give a more explicit construction in the specific situation of embedding $\Z \wr (\Z_2 \wr \Z)$. Section~\ref{sec:Iteration} gives a simpler construction for $\Z \wr \Z$, and illustrates why naive iteration of the construction in Lemma~\ref{lem:PointyWreath} does not lead to iterated permutational wreath products.

Theorem~\ref{thm:GinG} and Lemma~\ref{lem:PointyWreath} specifically say that we get wreath products $A \wr H$ from ``pointy'' actions of groups $H$, which simply means there is a free orbit on a configuration with finite $0$-support. Under a weaker condition on the action on finite points, we also obtain some permutational wreath products (Lemma~\ref{lem:WeakPointyWreath}).

The standard constructions of lamplighter-type groups (with finite base) and free groups are already pointy. We are in fact not aware of many other interesting pointy actions. One could answer Question~\ref{q:AreThe} in the positive by finding pointy actions for the top groups. (In which case one would even obtain embeddings for the variants with $\Z$ as the base group.) 

\begin{question}
Do the groups $\Z_2 \wr \Z^2$, $\Z \wr \Z$, or $\Z_2 \wr (\Z_2 \wr \Z)$ admit pointy actions, i.e.\ representations by elements of $\Aut(\Sigma^\Z)$ so that some configuration of finite $0$-support has free orbit?
\end{question}

We cannot even rule out that $\Aut(\Sigma^\Z)$ itself admits an abstract pointy action.

In Section~\ref{sec:Neumann}, we show that the simpler construction of Section~\ref{sec:Iteration} gives rise to some Neumann groups (from \cite{Ne37}) in $\mathcal{G}$, specifically the ones coming from eventually periodic sequences.

\begin{question}
Are there Neumann groups corresponding to aperiodic sequences in $\mathcal{G}$?
\end{question}

\section{Definitions}

The identity element of a group $G$ is $e_G$; $\N = \{0,1,2,\ldots\}$, $\Z_+ = \{1, 2,3, \ldots\}$. In a wreath product $K \wr H$ we usually refer to $K$ as the \emph{base} and $H$ as the \emph{top} group; such a group is a semidirect product $K^H \rtimes H$ where $H$ acts by the regular action. The free group on $n$ generators is $F_n$. If a group $G$ acts on a set $X$, an \emph{orbit} is $Gx = \{gx \;|\; g \in G\}$ for some $x \in X$, and a \emph{free orbit} is $Gx$ such that $g \mapsto gx$ is injective.

Throughout, $\Sigma$ is a finite discrete set called the \emph{alphabet}. The set $\Sigma^\Z$ is called the \emph{full shift}, and it is a compact dynamical system under the \emph{shift}, which is the homeomorphism $\sigma : \Sigma^\Z \to \Sigma^\Z$ defined by $\sigma(x)_i = x_{i+1}$. We write $\Z_n$ for the cyclic group with $n$ elements. The elements of $\Sigma^\Z$ are called \emph{points} or \emph{configurations}.

The set $\Sigma^+$ denotes all nonempty finite words over the alphabet $\Sigma$, i.e.\ the free semigroup on generators $\Sigma$ with concatenation as multiplication. Write $\Sigma^*$ for the corresponding free monoid adding the empty word. We write concatenation of $u, v \in \Sigma^*$ as $u \cdot v$ or just $uv$. We also work with left-infinite words $x \in \Sigma^{-\N}$ and right-infinite words $\Sigma^{\N}$, and can concatenate these together in various ways, with obvious interpretations, e.g.\ words from $x \in \Sigma^{-\N}, y \in \Sigma^{\N}$ and $w \in \Sigma^*$ join together into a point $x.wy \in \Sigma^\Z$, where $.$ is used to denote where the new origin is (let us say the symbol to the right of $.$ is coordinate zero).

A word or configuration over alphabet $\Sigma^k$ can be interpreted as $k$ \emph{tracks} of configurations over $\Sigma$. This simply refers to coordinatewise projection. Beware that for $k = 2$, we often instead want to interpret a word $w \in (\Sigma^2)^*$ as a \emph{conveyor belt}, as in Figure~\ref{fig:ConveyorBelt}, in which case the second track is read in reverse. We will use the term ``conveyor belt'' to clarify when this is done.

Sometimes we will define a notion that is really about positions in a configuration, such as defining a position $i \in \Z$ to be \emph{good} (in a configuration or word $x$) if it is contained in an interval $[a, b]$ such that the word $x_{[a, b]}$ has some property. We will in such discussions instead speak of particular symbols or subwords being good, and say things like ``the subword $v$ in $uvw$ is good'', meaning really that the interval where we explicitly write the subword $v$ in $uvw$ is good in whatever configuration, or \emph{context}, $uvw$ is being taken from. This should not cause confusion, and could be made precise by defining positioned words (which are converted into unpositioned words when necessary). % mainly causes confusion if one pays too much attention to it, and the reader is advised not to read this paragraph.

The finite set $\Sigma^{\Z_n}$ can be thought of as a $\Z$ or a $\Z_n$-system in an obvious way, with dynamics given by the map $\sigma$ defined by the same formula. For $n \in \Z_+$, the subset of $\Sigma^\Z$ of points $x$ satisfying $\sigma^n(x) = x$ is in natural correspondence with $\Sigma^{\Z_n}$. Such points are called \emph{$n$-periodic}. We write elements of $\Sigma^{\Z_n}$ the same way as we write words in $\Sigma^n$ (i.e.\ words $\Sigma^n$ of length $n$ over alphabet $\Sigma$).

The \emph{automorphism group of the full shift} $\Aut(\Sigma^\Z)$ consists of shift-commuting homeomorphisms $f : \Sigma^\Z \to \Sigma^\Z$, which we call \emph{automorphisms}. They are precisely the \emph{reversible cellular automata} i.e.\ the bijective maps which admit a finite \emph{neighborhood} $N \subset \Z$ and a \emph{local rule} $f_{\mathrm{loc}} : \Sigma^N \to \Sigma$ such that $f(x)_i = f_{\mathrm{loc}}(x_{i + N})$. The group $\Aut(\Sigma^\Z)$ is countable (due to local rules), residually finite (since periodic points are dense in $\Sigma^\Z$ and closed under $\Aut(\Sigma^\Z)$), and is not finitely generated \cite{BoLiRu88}.

Our alphabet $\Sigma$ typically contains a special symbol $0$ called \emph{zero}, which can be taken to be part of the structure. The \emph{nonzero} symbols are the ones that are not zero. Write $\Aut_0(\Sigma^\Z)$ for the point-stabilizer of $0^\Z$ in $\Aut(\Sigma^\Z)$. It is a finite-index subgroup, and contains an embedded copy of $\Aut(\Sigma^\Z)$ \cite{Sa22b}. A point $x \in \Sigma^\Z$ is \emph{$0$-finite} if $S = \{i \in \Z \;|\; x_i \neq 0\}$ is finite. Then $S$ is the \emph{$0$-support} of $x$.

%$\sigma$, also on $\Sigma^{n}$, words are also elements of said; aut also acts there
%identity is e

\section{The main technical result}

\begin{comment}
\begin{definition}
%Let $\Sigma \ni 0$ and let $u \in \Sigma^*$ be a nonempty word with $u_0 \neq 0, u_{|u| - 1} \neq 0$. %Let $N \subset \N$ be an eventually periodic infinite set such that $\min N \geq u$, and define $x^n = u0^{n-|u|}$.
Let $H$ be the point stabilizer of $0^\Z$ in $\Aut(\Sigma^\Z)$. Then $H$ acts naturally on $(\bigcup_{n \in \Z_+} \Sigma^n, \sigma)$. Define $G$ to be the group of permutations on $X$ containing
\begin{itemize}
\item all (permutations corresponding to) elements of $H$,
\item for each $u \in \Sigma^+$ with $u_0 \neq 0, u_{|u| - 1} \neq 0$, the map that acts as left shift on the shift-orbit of $u0^{n-|u|}$ for all $n$, and as identity elsewhere.
\end{itemize}
\end{definition}
\end{comment}

\begin{definition}
Let $\Sigma \ni 0$ be a finite alphabet. A set of finite words $u_i \in \Sigma^+$ is $n_0$-safe if: the words $u_i$ begin and end with nonzero symbols, are distinct (in particular these two imply that $\ldots 000 u_i 000 \ldots$ have disjoint shift orbits), $\max |u_i| \leq n_0$ and we can uniquely deduce $u_i$ and $j$ from $\sigma^j(u_i 0^{n - |u_i|})$ for any $n \geq n_0$.
\end{definition}

We call $n_0$ the \emph{safety threshold}. Here, $u_i 0^{n - |u_i|}$ is seen as an element of $\Sigma^{\Z_n}$, and $\sigma$ is the cyclic shift. For example, for $u_1 = 101$, the set $\{u_1\}$ is not $4$-safe, since $1010 = u_i 0 = \sigma^2(u_i 0)$, but it is $5$-safe. In general, if $u_i$ are distinct words that begin and end with nonzero symbols, then $U = \{u_1, \ldots, u_n\}$ is $n_0$-safe for any $n_0 \geq 2\max_i |u_i| - 1$, in particular if all the $u_i$ have just one letter, $\{u_1, \ldots, u_n\}$ is $1$-safe.

\begin{definition}
Let %. Then $H$ acts naturally on 
$X = \bigcup_{n \in \Z_+} \Sigma^{\Z_n}$. Define $G$ to be the smallest group of permutations on $X$ satisfying the following:
\begin{itemize}
\item $G$ contains all (permutations corresponding to) elements of $\Aut_0(\Sigma^\Z)$, applied to elements of $\Sigma^{\Z_n}$ interpreted as $n$-periodic configurations, and %(without modifying the $\Z_n$-components),
%\item for each finite subset of $u \in \Sigma^+$ with $u_0 \neq 0, u_{|u| - 1} \neq 0$, the map that maps $(u0^{n-|u|}, k) \to (u0^{n-|u|}, k+1)$ for all $n \in \Z_+$, $k \in \Z_n$, and as identity elsewhere.
\item for any $n_0$, any $n_0$-safe finite set $\{u_1, \ldots, u_k\} \subset \Sigma^+$, permutation $\pi \in S_k$ and $n_1, \ldots, n_k \in \Z$, $G$ contains the map that acts trivially on $\Sigma^{\Z_n}$ with $n < n_0$, and on $\Sigma^{\Z_n}$ with $n \geq n_0$ acts by
\[ \sigma^j(u_i 0^{n - |u_i|}) \mapsto \sigma^{j + n_i}(u_{\pi(i)} 0^{n - |u_{\pi(i)}|}). \]
for all $j$ (and trivially on words not of this form).
\item for any $n$, $G$ contains all permutations of $\Sigma^{\Z_n}$ (that do not modify any other elements).
\end{itemize}
\end{definition}

Note that $G$ fixes the sets $\Sigma^{\Z_n}$, and so is obviously residually finite. We often refer to the elements of $\Sigma^{\Z_n}$ as \emph{tapes}. Note that the first item gives an embedding of $\Aut_0(\Sigma^\Z)$ in $G$, because periodic points are dense in $\Sigma^\Z$. We write $H$ for the image of this embedding.

Elements of the form described in the second item are called \emph{AFO} (from ``actions on finitely many orbits''). We refer to the numbers $n_i$ as \emph{offsets}. The elements in the third item are called \emph{initial permutations}.

\begin{example}
For example, pick $\Sigma = \{0, 1, 2\}, u_1 = 1, u_2 = 2, n_0 = 1, \pi = (1 \; 2), n_1 = 1, n_2 = -1$. Then the AFO $f$ corresponding to this data maps
\[ f(1) = 2, f(20) = 01, f(12) = 12, f(00010)  = 00200, f(001000100) = 001000100. \]
Usually, we do not give the data for AFOs explicitly, but simply explain in words which configurations are shifted or modified, and how.
\qee
\end{example}

AFOs are a priori quite different from automorphisms, as they can in a sense read arbitrarily many symbols of a configuration in order to decide whether or not to modify or shift it, while automorphisms admit local rules. More generally, it may seem that compactness is a fundamental obstacle to performing such feats. The trick around this is simple: we add new symbols that allow us to explicitly mark infinite all-zero tails.

\begin{lemma}
The embedding $\Aut_0(\Sigma^\Z) \to G$ is split.
\end{lemma}

\begin{proof}
Recall that being split means there is a retraction, i.e.\ a homomorphism from $G$ to $\Aut_0(\Sigma^\Z)$ that maps trivially on the subgroup $\Aut_0(\Sigma^\Z)$. Consider a product $g$ of elements of $H$, AFOs and initial permutations. If we pick a large $n$ and a uniformly random element of $\Sigma^n$, then the initial permutations do not act, and with high probability none of the AFOs act (since the support of an AFO is of polynomial size in $n$).

In particular, for large enough $n$, $g$ eventually acts by a fixed automorphism in $H$ on more than half of the inputs. Such an automorphism $g'$ is obtained from the product representation of $g$ by dropping all the AFOs and initial permutations. It is also clear that for any other element $h \in G$, eventually it acts differently from $g'$ (and $g$) on almost all elements of $\Sigma^n$ (since with high probability we see all possible contents of the neighborhoods of both automorphisms). The retraction is then the map that takes $g$ to the unique $g'$ by which it eventually acts with probability greater than $1/2$.
\end{proof}

\begin{theorem}
\label{thm:GinG}
$G \in \mathcal{G}$.
\end{theorem}

\begin{proof}
We first perform the construction with $X = \bigcup_{n \in 2\Z_+} \Sigma^{\Z_n}$ (so we embed the group described exactly like $G$ above, but with $X$ only containing even-length tapes), and explain how to modify it for the union over $\Z_+$.

As our first step, we construct a new action of $\Aut_0(\Sigma^\Z)$ on a different full shift. As usual, we will use the idea of ``conveyor belts'' \cite{Sa18d} meaning we simulate the action of $\Aut_0(\Sigma^\Z)$ on paths that are wrapped as in Figure~\ref{fig:ConveyorBelt}, the point being that two pieces of tape with different orientations can be glued together into a continuous tape if we discover some kind of a problem.

The new idea is to use \emph{conveyor belts with floating boundaries}. This means that we will allow (though cannot quite force) the tape to contain explicit markings that tell us where the $0$-support of a configuration lies on the conveyor belt. Thus, in some situations we are able to locally detect that we are dealing with precisely a particular configuration $\sigma^j(u 0^{n-|u|})$ and not just a very good approximation of it.

\begin{figure}
\begin{tikzpicture}[scale=0.75]
\draw (0,0) grid (7,2);
\node () at (0.5,0.5) {0};
\node () at (1.5,0.5) {1};
\node () at (2.5,0.5) {0};
\node () at (3.5,0.5) {0};
\node () at (4.5,0.5) {1};
\node () at (5.5,0.5) {0};
\node () at (6.5,0.5) {1};
\node () at (6.5,1.5) {0};
\node () at (5.5,1.5) {0};
\node () at (4.5,1.5) {1};
\node () at (3.5,1.5) {0};
\node () at (2.5,1.5) {0};
\node () at (1.5,1.5) {1};
\node () at (0.5,1.5) {0};
\node () at (7.6,1) {$\approx$};
\draw[rounded corners=12] (8.2,-0.25)--(15.3,-0.25)--(15.3,2.25)--(8.2,2.25)--cycle;
\draw[rounded corners=4] (9,0.75)--(14.5,0.75)--(14.5,1.25)--(9,1.25)--cycle;
\draw (8.2,1) -- (9,1);
\draw (9.25,1.25) -- (9.25,2.25);
\draw (10.25,1.25) -- (10.25,2.25);
\draw (11.25,1.25) -- (11.25,2.25);
\draw (12.25,1.25) -- (12.25,2.25);
\draw (13.25,1.25) -- (13.25,2.25);
\draw (14.25,1.25) -- (14.25,2.25);
\draw (9.25,-0.25) -- (9.25,0.75);
\draw (10.25,-0.25) -- (10.25,0.75);
\draw (11.25,-0.25) -- (11.25,0.75);
\draw (12.25,-0.25) -- (12.25,0.75);
\draw (13.25,-0.25) -- (13.25,0.75);
\draw (14.25,-0.25) -- (14.25,0.75);
\draw (14.5,1) -- (15.3,1);
\node () at (8.7,0.37) {\rotatebox{-45}{0}};
\node () at (9.75,0.25) {\rotatebox{175}{1}};
\node () at (10.75,0.25) {\rotatebox{180}{0}};
\node () at (11.75,0.25) {\rotatebox{180}{0}};
\node () at (12.75,0.25) {\rotatebox{180}{1}};
\node () at (13.75,0.25) {\rotatebox{185}{0}};
\node () at (14.8,0.37) {\rotatebox{240}{1}};
\node () at (14.8,1.63) {\rotatebox{300}{0}};
\node () at (13.75,1.75) {\rotatebox{-5}{0}};
\node () at (12.75,1.75) {1};
\node () at (11.75,1.75) {0};
\node () at (10.75,1.75) {0};
\node () at (9.75,1.75) {\rotatebox{5}{1}};
\node () at (8.7,1.63) {\rotatebox{45}{0}};
%\node () at (7.6, -3) {$\approx$};
%\draw (10.6, -3.5) circle (2.5);
\end{tikzpicture}
\caption{A word of length $n$ over $\Sigma^2$ with $\Sigma = \{0,1\}$ can be interpreted as a conveyor belt containing a word of length $2n$, by concatenating the word on the top track to the reversal of the word on the bottom track.}
\label{fig:ConveyorBelt}
\end{figure}
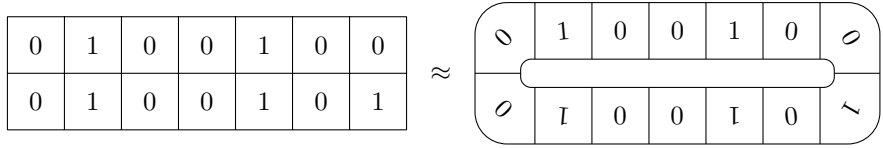

%which works on conveyor belts, in such a way that in some situations, we can ``see'' an action on finite-support points, with the $0$-tails explicitly marked. In this situation, we will be able to perform item $2$, i.e.\ shift an entire individual configuration in $F$ by a local rule, giving the cyclic generator. (Items $1$ and $3$ will be trivial to implement.) %, and this total shift is what gives us a copy of the natural action $\Z \wr G$ (the case of finite cyclic groups being a little easier).

%We may assume $x_0 = ...000.1000...$ by conjugating suitably, using [BoLiRu88].

We use the alphabet $\Gamma = \Sigma^2 \cup \{<, >\}$. Write $\mathbf{0} = (0, 0), B = \Sigma^2, C = B \setminus \{\mathbf{0}\}$, where $\mathbf{0}$ is thought of as the zero symbol of $\Gamma$. We will refer to a symbol $>$ (resp.\ $<$) as a \emph{wall} if the symbol to its left (resp.\ right) is \textbf{not} $>$ (resp.\ $<$). A \emph{good subword} is any word in
\[ \{>>, {>}C, BB, C{<}, <<, ><\}. \]
Here and below we abuse notation, and by $BB$ denote any word in $BB = \{bb' \;|\; b, b' \in B\}$; and similarly for ${>}C$ and $C{<}$. Other words of length $2$ are called \emph{bad}; explicitly these are
\[ \{{<}{>}, B{>}, {<}B, {>}\mathbf{0}, \mathbf{0}{<}\}. \]

A symbol that is only part of good words (i.e.\ the two words of length two that contain it are both good) is called \emph{good}. If a symbol is part of a bad word $ab$, and it is not a wall then it is called \emph{error}. %, so >0
%B> (then the B has error and > has wall)
%<B (then the B has error and < has wall)
%0<

%As a preprocessor step, sort of, we apply a block map that writes $\#\#$ on top of any occurrence of $<a$ for $a \neq <$, and also $a>$ where $a \neq >$. Also, we write $\#\#$ on top of $a<$ and $>a$ for $a \notin \Sigma \setminus \{0\}$. Then remaining $>$ and $<$ are thought of as $0$.

%Now consider the image of this configuration. We just write that on the tape, except: At the left end of a zone we have $>^n a$ and we pick the new $n$ by how much the zone extends to the left (or if it shrinks we increase $n$). Same on the right.

Let $f \in \Aut_0(\Sigma^\Z)$. We will describe the action of $f$ on doubly transitive points\footnote{Doubly transitive points are dense, so it suffices to describe the action on these points.} in such a way that it simulates the natural action of $\Aut_0(\Sigma^\Z)$ on encoded configurations. This action works as follows. No bad symbol (i.e.\ error or wall) is modified. Good symbols are part of maximal \emph{good runs}, i.e.\ words $v$ where all symbols are good. Explicitly, good runs are words of the form %${>}^* u {<}^*$ where $u \in B^*$, and either $u$ is the empty word, or the first and last symbols in $u$ are nonzero. Concretely,
%$v$ is in 
\[ {>}^* CB^*C {<}^* \cup {>}^* C {<}^* \cup {>}^* {<}^*, \]
and there are also requirements on the symbols around the word since the length-2 words on the boundaries must be good. %, namely:
%\begin{itemize}
%\item if the leftmost symbol is $>$, then the symbol to the left must be $>$ as well, 
%\item if the leftmost symbol is $\mathbf{0}$, then

We deal with such words differently depending on whether the symbols to the left and right are wall or error. Consider first the case where the symbols on the left and right are walls, i.e.\ the context is $a{>} \cdot {>}^m u {<}^n \cdot {<}b$ where $a \neq {>}$, $b \neq {<}$.

Our action of $\Aut_0(\Sigma^\Z)$ will rewrite the subword ${>}^m u {<}^n$. If $|u| = 0$, the action of $f$ is trivial. Otherwise, interpret the symbols $>$ and $<$ as $\mathbf{0}$s, and then read the resulting word $w = \mathbf{0}^m u \mathbf{0}^n$ as a pair of words $s, t \in \Sigma^*$ by separately reading the word on the top track and bottom track, i.e.\ $s_i = (w_i)_1, t_i = (w_i)_2$. 

Now apply $f$ to the periodic point $(st^R)^\Z$ to obtain a point $(s't'^R)^\Z$ with $|s'| = |s|, |t'| = |t|$. Then reverse the process by writing $w' =\! \begin{array}{c} s' \\ t' \end{array}$ (i.e.\ the word defined by $w'_i = (s'_i, t'_i)$) as $w' = \mathbf{0}^{m'} u' \mathbf{0}^{n'}$ with $m', n'$ maximal. Observe that the left- and rightmost symbols in $u'$ are nonzero, and $u'$ is not the empty word (since automorphisms cannot map nonzero points to the zero point), so this representation of the word is unique. Finally write $v' = {>}^{m'} u' {<}^{n'}$ back onto the tape. Note that the maximal good run remains the same, as the surrounding words $a{>}$ and ${<}b$ will not be modified (as all of their symbols are bad).

% \{{<}{>}, B{>}, {<}B, {>}0, 0{<}\}

Next, suppose the symbol to the left of $v$ is an error, and the symbol to its right is a wall. The context is then $ab \cdot {>}^m u {<}^n \cdot {<}c$, where $c \neq {<}$ and $ab$ is one of the cases ${>} \mathbf{0}$, ${<}B$, $\mathbf{0} {<}$. In each case, we must have $m = 0$, as otherwise the leftmost $>$ would in fact be a wall, i.e.\ our situation is actually $ab \cdot u {<}^n \cdot {<}c$ with $u \in B^*$, and the last symbol of $u$ is nonzero, unless $u$ is the empty word.

If $ab = \mathbf{0}{<}$, then $u$ must in fact be the empty word. If $ab = {>} \mathbf{0}$, then $u$ cannot be the empty word, as then its first symbol ${<}$ of $u {<}^n$ is not good. If $ab = {<} \mathbf{0}$, the same is true. If $ab = {<}C$, then $u$ can be empty or nonempty.

Now, if $u$ is the empty word, we do not modify the word $u {<}^n$. Otherwise, we work exactly as above: we replace $<$-symbols by $\mathbf{0}$s (this time, there are no $>$-symbols); then we applying $f$ to the corresponding conveyor belt, then we rewrite the maximal $\mathbf{0}$-suffix as $<^*$. The difference to the previous case is that we leave the maximal $\mathbf{0}$-prefix intact and write any initial zeroes as zeroes back on the tape.

Observe that this means we write back a word ending in a nonzero symbol, followed by ${<}$ symbols (again because no nonzero point maps to the all-zero point under $f$), in particular we always write a symbol from $B$ to the right of $ab$, so all the symbols written on top of the good run will be good, since the surrounding bad symbols $ab$ and ${<}c$ are again not modified.

If there is an error on the right, and a wall on the left, then we work as above, up to obvious left-right-symmetry. If there is an error on both sides, then the situation is $ab \cdot u \cdot cd$ where $u \in B^*$. We work exactly as in the previous cases, but rewrite neither the prefix nor the suffix of $\mathbf{0}$s into $>$ or $<$ in the last step, and just write the $f$-image back on the tape as such.

All in all, our description leads to an action of $\Aut_0(\Sigma^\Z)$, by automorphisms. Bijectivity is clear, since we are simply applying bijective transformations to simulated configurations through a natural interpretation that stays consistent between applications, so $f$ and $f^{-1}$ will be mapped to their inverses under this construction. Shift-commutation is obvious, and for continuity one simply checks that to determine the new symbol at the origin one only has to read $O(1)$ elements to determine the local conveyor belt structure, and apply $f$ to the simulated configuration; determining the final $>$ and $<$-rewrites is equally local, since $f$ fixes the point $0^\Z$.

As for AFOs, in situations where the conveyor belt has walls on each side (the first case of the four we considered), we can locally detect the entire configuration. Thus, we are able to shift around and permute any finite set of configurations.

%Next, the shifts on orbits of $u 0^{n-|u|}$ are easy to implement: we follow the same scheme, but instead of applying an automorphism, we only shift the encoding of a single orbit through the conveyor belt interpretation: ${>>}\begin{array}{c} 0u \\ 0^{|u|+1} \end{array}{<}$ is mapped to ${>>}\begin{array}{c} u0 \\ 0^{|u|+1} \end{array}{<}$, and on the left and right ends of a conveyor belt, we rotate $u$ around the corner, coordinate by coordinate to the lower track, where it is shifted in the other direction.
%It is easy to see that this is indeed an action by the edescribed group: the conveyor belts encode configurations $2\Z_+$.

Finally, to replace $2\Z_+$ by $\Z_+$, we can simply add a preprocessing step where we replace the neighborhoods $N \subset \Z$ of elements of $\Aut_0(\Sigma^\Z)$ by $2N$, so that $\Aut_0(\Sigma^\Z)$ effectively acts on two independent copies of a full shift. In AFOs we can replace all the words $u_i$ by words $u_i'$ where $0$s have been inserted between any two symbols of $u_i$ and at the end, and double the values $n_0$ and $n_1, \ldots, n_k$. In initial permutations, we replace a permutation $\pi$ of $\Sigma^n$ by a permutation of $\Sigma^{2n}$ that acts trivially when the encoded configuration's support intersects both cosets of $2\Z_n$, and otherwise acts as $\pi$ on the word on the unique coset it intersects. 

The subgroup obtained from the above construction is isomorphic to $G$: consider the new action on good runs of length $n$, i.e.\ a conveyor belts simulating periodic configurations of length $2n$, possibly with $>$- and $<$-affixes. The $G$-orbits where both the odd and even cells contain nonzero symbols will simulate natural $\Aut_0(\Sigma^\Z)$-orbits since we will never detect the shift-orbits of any $u_i'0^{2n-|u_i'|}$ and we explicitly cancel the initial permutations, so by the previous lemma we get only the relations in $G$. If the odd positions contain only zeroes, then they continue to do so as elements of $H$, AFOs or initial permutations are applied, and we are precisely simulating the action of $G$ on a single tape of size~$n$. %add $\Sigma$ to the alphabet, and allow the leftmost symbol on a conveyor belt to come from $\Sigma$ instead of $B$, with $0 \in \Sigma$ behaving exactly like $0 \in B$ as far as goodness of words is concerned. The conveyor belt then wraps according to Figure 2.
\end{proof}

\section{Pointy actions}

\begin{definition}
Let $H \leq \Aut_0(\Sigma^\Z)$. We say $H$ is \emph{pointy} if some $0$-finite configuration called the \emph{special point} $x_0 \in \Sigma^\Z$ has free orbit under $H$. %called the \emph{special point}, for any finite nonempty $F \subset G$ there exists $x \in \Sigma^\Z$ such that $gx = x_0$ for a unique element $g \in F$.
We say $H$ is \emph{weakly pointy with special point} $x_0 = \ldots 000u000\ldots$ if %the action on the orbit of $x_0$ is faithful, and
the point stabilizer of $x_0$ is contained in the point stabilizer of $u0^{n - |u|}$ for large enough $n$ (for the natural action of $H$ on $\Sigma^{\Z_n}$).
\end{definition}

%In our applications, we only construct pointy actions. Weak pointiness might potentially lead to other interesting examples, but we do not even have examples of interesting strictly weakly pointy actions. Rather, our reason to introduce this notion is that it clarifies what precisely is needed for a wreath product to arise out of the construction.

Note that a weakly pointy action is indeed a lot weaker than a pointy one, for example one can check that if $H$ is finitely generated and fixes the point $x_0$, then its action is weakly pointy for this special point. 

For a group $H \leq \Aut(\Sigma^\Z)$ and any point $x_0$, write $A \wr_{x_0} H$ for the permutational wreath product $A \wr_{\Omega} H$ where $\Omega = Hx_0$ and $H$ acts naturally on $\Omega$. Recall that such a wreath product is just the semidirect product of $A^\Omega \rtimes H$ where $h \in H$ acts on $y \in A^\Omega$ by $h y_{gx_0} = y_{h^{-1}gx_0}$. %We use a similar convention for subsets of $\Sigma^{\Z_n}$.

%We first prove a simple recoding fact about pointy actions.

\begin{comment}
\begin{lemma}
For a group $G \in \mathcal{G}$, TFAE:
\begin{itemize}
\item $G$ admits a pointy action on full shift $\Sigma \ni 0, 1$, with special point $x_0 = \ldots 000.1000 \ldots$.
\item $G$ admits a pointy action,
\item for some automorphism action of $G$ on a full shift $\Sigma^\Z$, for some periodic point $p$, the action of $G$ on the points doubly asymptotic to $p$ has a free orbit,
\end{itemize}
\end{lemma}

\begin{proof}
The downward implications are clear. To get from the third to the second item, take the $p$th higher power shift and rename a letter to $0$ to get $x_0$. To get from the second to the first item, conjugate $x_0$ to the stated form, using the fact the action of $\Aut(\Sigma^\Z)$ is transitive on the homoclinic points [BoLiRu88].
\end{proof}
\end{comment}

\begin{lemma}
\label{lem:WeakPointyWreath}
Suppose $H \leq \Aut_0(\Sigma^\Z)$ is weakly pointy with special point $x_0$. Then $A \wr_{x_0} H \in \mathcal{G}$ for all finitely generated abelian groups $A$.
\end{lemma}

\begin{proof}
It suffices to show $A \wr_{x_0} H \in \mathcal{G}$ for finite abelian $A$ and infinite cyclic $A$, since $(N \times K) \wr_{\Omega} H \leq (N \wr_{\Omega} H) \times (K \wr_{\Omega} H)$. We start with the case of a finite $A$, whose group operations we will write additively. % and then explain how to get the much easier $A = \Z_n$. %Recall that $A \wr G$ admits a natural action on $A \times F$ as follows: The subgroup $G$ acts by $g \cdot (a, x) = (a, gx)$ and elements $t \in A^F$ act by $t \cdot (a, x) = (t_x + a, x)$.

% $(N \times K) \wr_{\Omega} H \leq (N \wr_{\Omega} H) \times (K \wr_{\Omega} H)$ proof
% we have elements (n, k)^x and h on the left
% map to ((n^x, 1), (k^x, 1)) and (h, h) on the right; product on each side,
% and conjugation action, work the same way

Suppose $H$ acts weakly pointily with special point $x_0 = \ldots 000 . u 000 \ldots$. Take alphabet $\Sigma \times A$ with zero element $(0, 0_A)$, and have $H$ act on the first track, ignoring the second. Now consider the subgroup of $G$ generated by $H$, and the AFOs that take $(u0^{n-|u|}, a0^{n-1})$ to $(u0^{n-|u|}, (a + a')0^{n-1})$ for various $a'$, for all $n \geq n_0$ (but do nothing else).

We show that these maps form a copy of $A \wr_{x_0} H$ inside $G$. Observe that this abstract group is generated by $H$ and generators of $A$ applied ``at'' (approximations of) $x_0$, and the conjugation action of $H$ is to move the points by its natural action (carrying with them the element of $A$ that has been applied on each point).

First, consider $G$ as embedded in the automorphism group of a full shift according to the proof of Theorem~\ref{thm:GinG}. Consider tapes where $x_0$ appears in the form ${>}^{-\N} ((u, a0^{|u|-1}), 0^{|u|}) {<}^{\N}$. Clearly on such points we are simulating in a concrete way this abstract action of the group: $H$ literally acts by its natural action on the top track, and the generators for the bottom $A$ (the AFOs) add $a'$ to the unique symbol on the bottom track, if the top track contains precisely the configuration $x_0$. Thus, the group we have constructed inside $G$ admits an epimorphism onto $A \wr_{x_0} H$.

Then again consider $G$ abstractly. By a similar analysis as in the previous paragraph, the action of the subgroup of $G$ we constructed, on the set $\Sigma^{\Z_n}$, is just the natural action of $A \wr_{u 0^{n-|u|}} H$. By the assumption on stabilizers, the action $H \curvearrowright O_H(u 0^{n-|u|})$ is a factor of $H \curvearrowright O_H(x_0)$ (by mapping $h \cdot x_0 \mapsto h \cdot u0^{n-|u|}$. Thus $A \wr_{u 0^{n-|u|}} H$ is a factor of $A \wr_{x_0} H$ (since $A$ is abelian). We conclude that $A \wr_{x_0} H$ factors onto the permutation groups by which $G$ acts on $\Sigma^{\Z_n}$.

%: first, we claim that the mapping thus induced is a homomorphism. Namely suppose that we represent a trivial element of $A \wr H$ by composing these generators. In particular the $H$-projection is trivial so on the side of $G$ the first track returns to its initial contents. On the other hand for any element in the orbit of an $x_n$.

% Thus $A \wr H \in \mathcal{G}$.
For $\Z \wr_{x_0} H$, we do the same, but now instead of acting on $\Sigma \times A$ we act on $\Sigma \times \{0,1\}$. We use the natural action of $H$ on the first track, and then in $G$ we consider the subgroup containing $H$ and the AFO that shifts points from $O((u0^{n-|u|}, 10^{n-1}))$ by offset $1$, and fixing other points.

Again, on infinite conveyor belts (in the sense of the embedding of $G$ into the automorphism group of a full shift) we have the natural action of $\Z \wr_{x_0} H$: when shifting on the base group $\Z$ we shift a unique bit on the second track. As a side-effect we shift the entire configuration, but the action of $H$ commutes with the shift so we can ignore this and only keep track of the movement on the second track: the total shift on the first track is invisible to the action of $H$, and its total shift exactly copies the total shift of the second track.

On finite conveyor belts, by a similar analysis we have an action of the group $\Z_n \wr_{u 0^{n-|u|}} H$, which is a factor of $\Z \wr_{x_0} H$ by the assumption on stabilizers.
\end{proof}

%There are of course many interesting modifications one can make in the construction % In particular, one can use the special points as ``heads of Turing machines'', and modify a second tape under these heads, to get something resembling a wreath product of the automorphism group and a group of Turing machines. The finite tapes ultimately lead to a very complicated group, and we have not found any particular applications for this construction.

\begin{comment}
It is easy to see that the finite index group $G \leq \Aut(\Sigma^\Z)$ acts faithfully on the $0$-finite points of $\Sigma^\Z$. However, it does not act freely. By finding subgroups where this action admits free orbits, we obtain wreath products of groups. We introduce some terminology for this, and make some dynamical/recoding observations.

\begin{definition}
Let $G$ be a group. Let $\Sigma \ni 0$ be a finite alphabet. An action of $G$ by automorphisms on $\Sigma^\Z$ is \emph{pointy} if some $0$-finite configuration called the \emph{special point} $x_0 \in \Sigma^\Z$ has free orbit. %called the \emph{special point}, for any finite nonempty $F \subset G$ there exists $x \in \Sigma^\Z$ such that $gx = x_0$ for a unique element $g \in F$.
\end{definition}

%Note that if the orbit of $x_0$ under $G$ is free, then the action is automatically pointy, as we can pick $g \in F$ arbitrarily and $x = g^{-1} x_0$ in the previous definition. In all of our examples of pointy actions, this will be the case.

\end{comment}

\begin{lemma}
\label{lem:PointyWreath}
Suppose $G \in \mathcal{G}$ admits a pointy action. Then $A \wr G \in \mathcal{G}$ for all finitely-generated abelian groups $A$.
\end{lemma}

\begin{proof}
A pointy action is weakly pointy, and we have $A \wr_{x_0} G \in \mathcal{G}$. Since the orbit is free, this is simply the group $A \wr G$. %simply means that the group acts freely on the orbit of some $0$-finite point $x_0$. A permutational wreath product $A \wr_F G$ has an isomorphic copy of $A \wr G$ as a subgroup, whenever $G \curvearrowright F$ has a free orbit.
\end{proof}

\section{Examples of pointy actions}

\begin{lemma}
The set of groups admitting pointy actions is closed under direct products, subgroups, and commensurability.
\end{lemma}

\begin{proof}
For direct products, we use different tracks, one for each group, and the action $(g, h)(x, y) = (gx, hy)$. As the special point $x_0$ we use the pair of special points of the two actions, one on each track.

For subgroups, obviously for any pointy action, the subaction of any subgroup is also pointy.

For commensurability, since we have closure under subgroups, it suffices to show that if $G$ is of finite index in $H$ and $G$ admits a pointy action, then so does $H$. %We can assume $G$ is normal and of finite index by taking the kernel of the action of $H$ on its cosets. Then
The Krasner–Kaloujnine universal embedding theorem states that $H$ embeds in $G \wr S_n$ for some $n$ ($S_n$ is the symmetric group on $n$ points). Since we have direct products, we have a pointy action of $G^n$, and we can combine this with the action of swapping the contents of various tracks to implement the action of $S_n$ (it is easy to keep the action free).
\end{proof}

%Now we look at free groups. Up to commensurability, a nonamenable free group is the same thing as a free product of finite groups, 

\begin{lemma}
Every countable free group admits a pointy action.
\end{lemma}

\begin{proof}
It suffices to consider the free group on two generators. Free nonamenable groups on finitely many generators are pairwise commensurable, and by a theorem of Nielsen \cite[Theorem~2]{Ly73} they are the same as finite free products of finite groups up to commensurability, so it suffices to consider finite free products of finite groups.

Alperin constructs an action for such groups in \cite{Al88}. His proof in fact shows that this action is pointy. Specifically, the action on $\ldots111{*}111\ldots$ is free (i.e.\ we take the identity element as the zero symbol). (One can equivalently use the construction from \cite{BoLiRu88} which is for free products of $\Z_2$s.)
\end{proof}

% In particular it follows from this lemma that free groups admit pointy actions, namely the free product of two nontrivial finite groups, one of which is not $\Z_2$, is commensurable to $F_2$.

\begin{lemma}
Let $A$ be a finite abelian group. Then the group $A \wr \Z$ admits a pointy action.
\end{lemma}

\begin{proof}
The action constructed for these groups in \cite{Sa20c} (explained also in the example below) is easily seen to be pointy. The orbit of the configuration $\begin{array}{c}\ldots 0001000 \ldots \\ \ldots 0000000 \ldots\end{array}$ is free.
\end{proof}

\begin{example}
\label{ex:ZZ2Z}
We illustrate the embedding of $\Z \wr (\Z_2 \wr \Z)$ obtained by combining Lemma~\ref{lem:PointyWreath} with the previous lemma. We will refer to $\Z_2 \wr \Z$ as the \emph{top lamplighter}, and the $\Z$ on the left as the \emph{bottom $\Z$}. The generators of the top lamplighter will be called $L, R$ (left and right; the generators corresponding to a symmetric generating set for the $\Z$ quotient) and $F$ (flip; the generator of $\Z_2$ at $0_\Z$), and $U, D$ (up and down) are the generators for the bottom $\Z$ at the identity of $\Z_2 \wr \Z$.

We use the usual action of $\Z_2 \wr \Z$ on $(\{0,1\}^2)^\Z$ (see e.g.\ \cite{Sa20c}): the generator of $\Z$ shifts the first track, and the generator of $\Z_2$ (at $0_\Z$) sums the first track to the second modulo $2$. %, we act on $(\{0,1\}^3)^\Z$ and never modify the second track. This is not strictly necessary, and indeed in the proof of Lemma~\ref{lemma:PointyWreath} we do not add such an additional track, but this way we can see the difference between head moves of the top lamplighter $\Z_2 \wr \Z$ and total shifts of finite configurations corresponding to the bottom $\Z$.

The way we implement $\Z \wr (\Z_2 \wr \Z)$ in $G$ is now that we take the above copy of $\Z_2 \wr \Z$ in $H$. Then we modify its action according to the proof of Theorem~\ref{thm:GinG} so that it acts on conveyor belts with floating boundaries. The generator $D$ for the bottom $\Z$ (at the origin of top lamplighter) is implemented by shifting the word $(1,0,1)$ (seen as a word of length $1$, over the alphabet $\{0,1\}^3$) by $1$, i.e.\ the AFO data is $u_1 = (1,0,1)$, offset $n_1 = 1$, safety threshold $n_0 = 1$; $U$ is its inverse.

We illustrate the action of this group in Figure~\ref{fig:SpacetimeDiagram}, by computing a spacetime diagram for the element $(FL)^3 \cdot ULUFRD^4LFR$, in the sense of showing how partial applications of this product modify a particular configuration (which we have chosen so that it illustrates some of the relevant phenomena). The symbols $>, <$ are directly written as themselves, while we write an element of $(\{0,1\}^3)^2$ as a stack of three colored boxes, the top three boxes corresponding to the first $\{0,1\}^3$ (the top track) and the bottom three boxes corresponding to the latter $\{0,1\}^3$ (the bottom track). The fill color white corresponds to $0$, and all other colors correspond to $1$.

There are three good runs visible. On the leftmost the top lamplighter configuration corresponds to the identity, so $ULUFRD^4LFR$ adds $1$ to the bottom copy of $\Z$, and then $(FL)^3$ flips a few more bits on the top lamplighter side. In the middle conveyor belt, the top lamplighter is in state $LFR$, so the initial $ULUFRD^4LFR$ instead subtracts $4$ from the bottom copy of $\Z$. Note that now $(FL)^3$ moves to the right, being on the bottom tape.

On the third tape we just see the action of $(FR)^3(LFR)^2$ of the top lamplighter for three reasons: there is an error on the right (so the right end does not float), there are two top lamplighters (i.e.\ $1$s on the simulated top track), and there are two $1$s on the simulated bottom track (corresponding to the bottom $\Z$). Note that between the second and third good run there is also a zone with only errors, which is never modified. \qee
\end{example}

\begin{figure}
\begin{tikzpicture}[scale=0.5]
\input{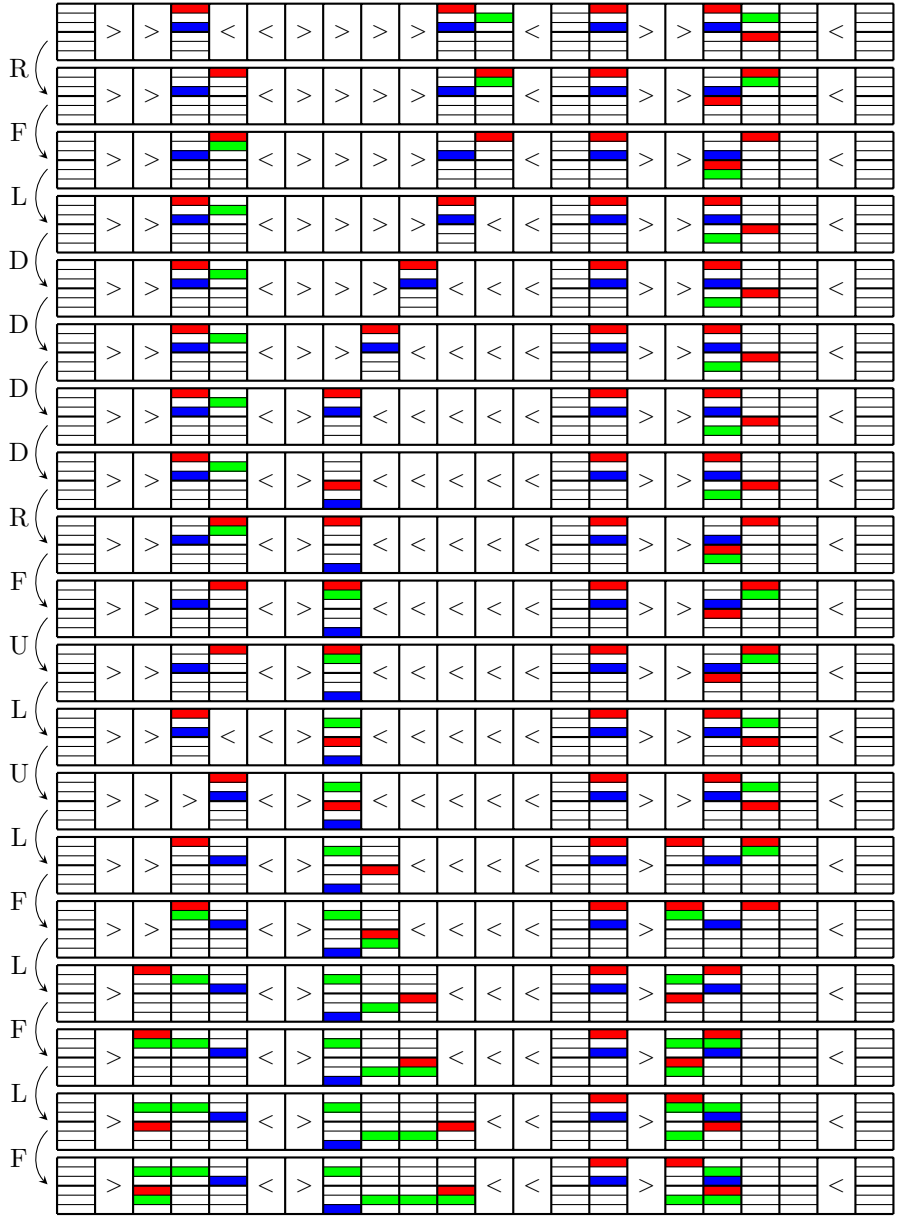}
\end{tikzpicture}
\caption{A spacetime diagram for the element $(FL)^3 \cdot ULUFRD^4LFR$. }
\label{fig:SpacetimeDiagram}
\end{figure}

%\begin{figure}
%\begin{tikzpicture}[scale=0.5]
%\input{orbitrand}
%\end{tikzpicture}
%\end{figure}

The Python script at \cite{Sa23} implements the action of $\Z \wr (\Z_2 \wr \Z)$. When run, it generates Tikz code for the figure in the example.

%\begin{lemma}
%Consider $\Aut(\Sigma^\Z)$. It admits an action on the finite points, and also actions on periodic points. 
%\end{lemma}

%In \cite{Sa}, we worked out a basic theory of embedding wreath products, and showed in particular that if a group $G$ acts by automorphisms of $B^\Z$ in such a way that for each $1_G \neq g \in G$ there exists a point $x \in B^\Z$ and a clopen set $C$ with bounded diameter independent from $x$ such that $x$ enters $C$ during its orbit only at the $g$-translate, then $A \wr G \in \mathcal{G}$. This gives that $A \wr \Z^d \in \mathcal{G}$, and also examples like $A \wr F_2$ ($F_2$ being the free group on two generators), for finite $A$, but this approach does not generalize to infinite groups $A$.

The following is immediate from the closure properties and constructions in this section. Combining it with Lemma~\ref{lem:PointyWreath} immediately yields the results mentioned in the introduction.

\begin{lemma}
\label{lem:PointyConstructions}
If $B$ is a finite abelian group, $C$ is a finitely-generated free group, and $n \in \N$, then $A \wr ((B \wr \Z)^n \times C^n)$ is in $\mathcal{G}$.
\end{lemma}

\section{A simpler construction that gives $\Z \wr \Z$ and some other groups}
\label{sec:Iteration}

%As we present the construction above, it cannot directly be ``iterated'' to construct $\Z \wr (\Z \wr \Z)$, or even groups of the form $\Z \wr_{\Omega} (\Z \wr \Z)$ with interesting actions on $\Omega$, by starting with a pointy action of $\Z$ and applying Lemma~\ref{lem:WeakPointyWreath}. The issue is that the action of $\Z \wr \Z$ that we produce is not pointy, since the ``special points'' where the action is interesting (and not just one of the top group) are of the form ${>}^{-\N} u {<}^\N$, so not of finite support.

%It is in a sense possible to extend the ``conveyor belts with floating boundaries'' construction to such points, by having the boundaries $>$ and $<$ represent different tails, and replacing the conveyor belt turns at the ends of a good run with blocking words (though we omit the details). There are, however, two more fundamental problems: it is not clear how one could get a $\Z$-base for the wreath product in this generality (because we cannot rotate the support around a belt), and while one always can get something resembling a wreath product with finite base, the action on finite belts seems to almost never end up being a homomorphic image of the group acting on infinite tapes, only an approximation of it in the space of marked groups. The latter phenomenon leads to the final group typically being somewhat more exotic than a permutational wreath product.

We give here a simpler variant of the construction from Theorem~\ref{thm:GinG}, which does nothing more than marks individual positions on conveyor belts (and has no explicit conveyor belts). In Section~\ref{sec:Solvable}, we give an algebraic description of some solvable groups obtained from, in a sense, iterating Lemma~\ref{lem:WeakPointyWreath} for this construction, illustrating how the final group differs from a standard iterated permutational wreath product. In Section~\ref{sec:Neumann}, out of general interest, we show how this construction gives rise to some Neumann groups in $\Aut(\Sigma^\Z)$.

\subsection{The simplified construction}

The construction roughly follows that of Theorem~\ref{thm:GinG}, and the reader should be familiar with the proof. We consider the alphabet $\{1,2,{>},{<}\}^k$, $C = \{1, 2\}$. In a point of $\{1, 2, {>}, {<}\}^\Z$, we designate $\{{>}{>}, {>} C, C {<}, {<}, {<}\}$ as \emph{good words}, while the remaining words $\{CC, {<}C, C{>}, ><, <>\}$ are \emph{bad}.

A coordinate in a point of $\{1, 2, {>}, {<}\}^\Z$ is \emph{good} if both words of length $2$ containing it are good, and in a point over the alphabet $\{1,2,{>},{<}\}^k$, a coordinate is \emph{good} if it is good on all the $n$ tracks individually. Now, each doubly transitive configuration splits into finite maximal good runs, meaning maximal runs where on each track we only have letters which are not contained in a bad word. A good run over the alphabet $\{1,2,{>},{<}\}^k$ will give a good run on each single track, and these runs are of the form ${>}^* C {<}^*$, ${>}^+$, ${<}^+$, with the left- and rightmost letters extending to a good word on each side. Of course, bad words on any one track will also cut good runs on all other tracks so that they may no longer be maximal on that track.

It is clear that on good words of length $n$ on an individual track, we have an action of $\Z_{2n}$ which moves $1$-symbols to the left, and $2$-symbols to the right, except at the left end it flips $1$ to $2$ and at the right end it flips $2$ to $1$. Again we can think of this as a conveyor belt, with a single $1$ moving to the left on the ``top track'' and $2$ moving to the left on the ``bottom track''. This action does not change the set of good subwords on the track it modifies, so it also does not change the set of good positions of the composite configuration over $\{1,2,{>},{<}\}^k$.

More generally, if we only make modifications to good runs ${>}^m a {<}^n$ by changing the values of $m$ and $n$, and possibly changing $a$ to $3-a$ (even depending arbitrarily on the other tracks!), we will not modify the global set of areas where good runs appear. From this one gets a large variety of different behaviors, in particular, one obtains a simple embedding of $\Z \wr \Z$. This, and some other constructions are explored in the following sections.

\subsection{Some solvable groups obtained}
\label{sec:Solvable}

%This is how we originally found the construction presented in this paper: it was already observed in \cite{SaSc16a} that $\Z \wr \Z$ embeds in the automorphism group of a countable sofic shift, and we found a way to mimic this action on periodic tapes.

Let us now look at what kind of group we get if each head only affects heads with higher index. This type of construction intuitively should give a solvable group, since taking the commutator subgroup effectively eliminates one head at a time.

\begin{definition}
Let $n \in \N$ and let $K_1^n = \Z/n\Z$ (i.e.\ $K_1^0 = \Z$, and $K_1^n = \Z_n$ for $n \geq 1$). Then $K_1^n$ acts on $\Z/n\Z$ by the regular (translation) action. By induction, $K_k^n$ acts on $(\Z/n\Z)^k$, and we define $K_{k+1}^n = \Z/n\Z \wr_{(\Z/n\Z)^k} K_k^n$, which acts naturally on $(\Z/n\Z)^{k+1}$: the top group $K_k^n$ acts on the first $k$ coordinates by its natural action, and the generator for the base $\Z/n\Z$ (the delta at $e_{K_k^n}$) acts by the regular action on the last coordinate if and only if other coordinates are $0$.
\end{definition}

%For example, $K_2^n = \Z/n\Z \wr K_1^n$ acts on $(\Z/n\Z)^2$ in an obvious way: the top $K_1^n$ acts on the leftmost $\Z/n\Z$ in $(\Z/n\Z)^2$ by its regular action, and the bottom $\Z/n\Z$'s generator (the delta at $0_{\Z/n\Z} = e_{K_1^n}$) acts by translation on the rightmost $\Z/n\Z$, if the leftmost $\Z/n\Z$-coordinate has value $0$; otherwise it does nothing. We obtain a permutational wreath product $K_3^n = \Z/n\Z \wr_{(\Z/n\Z)^2} K_2^n$ from this action, and again the resulting group acts naturally on $(\Z/n\Z)^3$. We can iterate this construction to obtain an iterated permutational wreath product $K_{k+1}^n = \Z/n\Z \wr_{(\Z/n\Z)^k} K_k^n$.

There is a natural generating set for $K_k^n$, where the generator $g_i$ adds $1$ to the $i$th coordinate if all smaller coordinates $\{1, 2, \ldots, i-1\}$ contain $0$ (modulo $n$, when $n \geq 1$). Using these generators, we can consider $K_k^n$ as elements of the space of marked groups on $k$ generators \cite{Gr84}. The following is easy to show.

\begin{lemma}
\label{lem:Limit}
$K_k^0 = \lim_n K_k^n$ in the space of marked groups.
\end{lemma}

\begin{definition}
Define $\hat K_k$ as the group generated by $k$ elements $g_1, \ldots, g_k$ with relations precisely the relations that hold in all of the $K_k^n$.
\end{definition}

\begin{theorem}
For all $k$, $\hat K_k \in \mathcal{G}$.
\end{theorem}

We show below that $\hat K_2 = \Z \wr \Z$, so this gives a simpler construction for this group in $\mathcal{G}$, as suggested in the previous section.

\begin{proof}
We use the construction from the previous section, on alphabet $\{1,2,{>},{<}\}^k$. As in the description above, we interpret good runs, which are either degenerate, or contain a unique \emph{head} which is an element of $\{1, 2\}$. To a head we associate its \emph{belt position} by $p(0^m 1 0^n) = m$, $p(0^m 2 0^n) = m+2n+1$.

The automorphism corresponding to $g_i$ adds $2$ to the belt positions on tracks $\{1, \ldots, i\}$ if and only if there are indeed heads on all these tracks, and their belt positions are equal. Concretely, all of the tracks should have a $1$-symbol in the same position, or all should have a $2$-symbol in the same position. (We move the heads' belt positions $2$ steps at a time simply so that we effectively get also tapes of odd length.) 

It is immediate that on a single good run, where all heads are present and in belt positions with the same parity, this mimics the action above. To conjugate the actions when the parities are even, the tuple of belt positions $(2h_1, \ldots, 2h_k)$ is mapped to $(h_1 - h_2, h_2 - h_3, \ldots, h_{k-1} - h_k, h_k)$. As in the main construction (for wreath products with base $\Z$), the point is that our automorphisms will also move the first $i-1$ heads when moving the $i$th, but in any case once the $i$th head returns back, this shift has been undone.

If $\ell$ is minimal such that a head is missing from the $\ell$th track in some good run (i.e.\ that tape is degenerate) or the belt position of the head has a different parity than head on the previous track, then on that particular good run we simply mimic the quotient action of $K^n_{\ell-1}$.
\end{proof}

The following proposition illustrates (for this specific example) that iterating the main construction does not lead to an iterated permutational wreath product ($K_k^0$), but instead to a product of finite iterated permutational wreath products ($\hat K_k$), and these can be different.

\begin{proposition}
The map $g_i \mapsto g_i$ gives a homomorphism from $\hat K_k \to K_k^0$.
This is an isomorphism if % $\hat K_k$ is equal to $K_k$ if
$k \leq 2$.
If $k \geq 3$, then $K_k^0$ is torsion-free, but $\hat K_k$ is not.
All of the groups $K_k^0$, $\hat K_k$ are solvable.
\end{proposition}

\begin{proof}
The first statement is clear from $K_k^0 = \lim_n K_k^n$. For the second, it is obvious that $\hat K_1 = K_1^0 = \Z$, and from Lemma~\ref{lem:Limit} it is clear that $K_2^0 \cong \Z \wr \Z$.

By the first statement, we have a homomorphism $\hat K_2 \to \Z \wr \Z$. We show that this is an isomorphism. Consider an element in the kernel, i.e.\ $g \in \hat K_2$ which is represented as a product of $g_1$ and $g_2$ whose $\Z \wr \Z$ action on $\Z^2$ is trivial. The generator $g_1$ is then applied equally many times forward and backward, so the same happens in the first coordinate of $\Z_n^2$ when $g$ is applied. Since $g$ acts trivially also in the second coordinate of $\Z^2$, for every $\ell$ we must apply $g_2^{g_1^\ell}$ equally many times forward and backward, so the same is true in $\Z_n^2$. Thus, the kernel is actually trivial, so we indeed have an isomorphism. % $\hat K_2 \cong K_2^0 \cong \Z \wr \Z$.

%Consider a non-trivial element of $\hat K_2$. It acts non-trivially on $\Z_n^2$ for some $n$. This action is nothing but the standard wreath product action of $\Z_n \wr \Z_n$, so either the top element is nonzero in $\Z_n$ (in which case the element certainly maps non-trivially to $\Z \wr \Z$) or it increments the value at some cell in $\Z_n$ by a nonzero value. Then the same must be true in the $\Z \wr \Z$ image. %it suffices to show that $\hat K_2$ satisfies all relations of $\Z \wr \Z$.
%mapping the generators of $K_2^0$ to those of $K_2^n$ gives a homomorphism. 

%hatK2 has LESS relations. but we want to show that has in fact all relations of \Z \wr \Z.

%and $\hat K_2 = K_2 = \Z \wr \Z$.

Next, it is easy to see that the groups $K_k^0$ are all torsion-free and $\hat K_{k+1}$ contains a copy of $\hat K_k$, so for the third statement, it now suffices to show that $\hat K_3$ is not torsion-free.

Let $n \geq 2$. We observe that $h = g_2^{g_1^n}$ stabilizes $(0, 0)$ in $K_2^0$, but does not stabilize $(0, 0)$ in the action of $K_2^n$. This means $g_3^h g_3^{-1}$ is trivial in $K_3^0$ but not in $K_3^n$. Since $K_3^0$ is the limit of the groups $K_3^\ell$, $g_3^h g_3^{-1}$ is trivial in $K_3^\ell$ for sufficiently large $\ell$. Thus its order is bounded over the finitely many finite groups $K_3^\ell$ where it acts nontrivially, and therefore its order is finite in $\hat K_3$.

For the last statement, %the group $\hat K_k$ is solvable of derived length at most $k$ by definition as an iterated permutational wreath product (one can check that indeed it is precisely $k$).
it suffices to show that $\hat K_k$ is solvable, as $K_k^0$ is its quotient. Consider for $N \subset \N$ the group $K_k^N$ where we take generators $g_i$ under precisely the relations that hold in all of the $K_k^n$. This is just the subgroup of $\prod_{n \in N} K_k^n$ generated by the $k$ many diagonal elements $(g_i, g_i, g_i, \ldots)$.
Now we observe that the entire group $\prod_{n \in N} K_k^n$ is solvable of derived length at most $k$, as a product of groups of derived length at most $k$. For this, recall that solvable groups of degree $k$ form a variety, or concretely observe that $[\prod_i G_i, \prod_i G_i] \leq \prod_i [G_i, G_i]$ for any groups $G_i$ by a direct computation. (One can check that the derived length of both $K_k^0$ and $\hat K_k$ is precisely $k$.)
\end{proof}

\subsection{Neumann groups}
\label{sec:Neumann}

Next, let us look at what we get when one has a single head.

\begin{definition}
\label{def:Neumann}
Let $3 \leq n_1 < n_2 < \ldots$ be positive numbers, and let $u_{i,j}$ be distinct elements for $1 \leq j \leq n_i$. A \emph{generalized Neumann group} corresponding to the sequence $(n_i)_i$ is the subgroup of the product of symmetric groups $\prod_i \Sym(\{u_{i,j} \;|\; 1 \leq j \leq n_i\})$ which is generated by the permutations $a = \prod_i (u_{i,1}; \; u_{i,2}; \; u_{i,3})$ and $b = \prod_i (u_{i,1}; \; u_{i,2}; \; \ldots; \; u_{i,n_i})$. If all the $n_i$ are odd, we call these \emph{Neumann groups}. % where $k_i = 2 + i - 2 \lceil i/2 \rceil$.
\end{definition}

%$2j+1 -> 1, 2j -> 2$
%Here $k_1, k_2, k_3, k_4 \ldots = 1, 2, 1, 2, \ldots$ is simply a way to make the permutations even.
Neumann defines the Neumann groups in \cite{Ne37}. Note that in Neumann groups, $a$ and $b$ act by even permutations, and thus their actions can be taken to be by elements of $\prod_i \Alt(\{u_{i,j} \;|\; 1 \leq j \leq n_i\})$. He shows that then two such groups are isomorphic if and only if the defining sequences are the same. %, in the case of odd $n_i$s. We include the even case, as %Since he assumes the numbers are odd, and the even case is
%the conveyor belt construction most naturally gives even permutations. For the same reason,
We include the proof of this fact. %the most natural one in our present context, we recall the proof.

\begin{lemma}
\label{lem:Determines}
Let $G$ be the Neumann group coming from a sequence $5 \leq n_1 < n_2 < \ldots$ of odd numbers. Then the isomorphism type of $G$ determines the sequence $(n_i)_i$. In fact, the sequence is determined by the isomorphism types of finite simple normal subgroups.
\end{lemma}

\begin{proof}
The element $[a, a^{b^{n+3}}]$ is nontrivial in $A_{n+5}$ (where $a = (1; \; 2; \; 3)$ and $b = (1; \; 2; \; \ldots; \; n)$), and trivial in $A_{n+m}$ for $m \geq 6$. By induction on $i$, it is then easy to show that the natural $A_{n_i}$ in $G$ (other $A_n$ acting trivially) is a subgroup of $G$, as once we have built $A_{n_{i-1}}$, the normal closure of $[a, a^{b^{n_i-2}}]$ is a subgroup of $\prod_{h \leq i} A_{n_i}$ surjecting to $A_{n_i}$, and we can cancel the smaller $A_{n_h}$. Furthermore, this $A_{n_i}$ is clearly normal.

On the other hand, consider any finite normal subgroup of $G$. Its elements clearly have nontrivial projection to only finitely many of the $\Alt(\{u_{i,j} \;|\; j \leq n_i\})$. It is then a straightforward consequence of Goursat's lemma \cite{Go89} that if such a group is simple, it is precisely one of the $\Alt(\{u_{i,j} \;|\; j \leq n_i\})$. All in all the $n_i$ are determined by isomorphism types of finite simple normal subgroups of $G$.
\end{proof}

\begin{lemma}
Let $3 \leq n_1 < n_2 < \ldots$ be any eventually periodic sequence of numbers, meaning for some $p \geq 1$ we have
\[ \exists n_0: \forall n \geq n_0: n \in \{n_i \;|\; i \in \Z_+\} \iff n+p \in \{n_i \;|\; i \in \Z_+\}. \]
Then the corresponding generalized Neumann group $G$ is in $\mathcal{G}$.
\end{lemma}

By Lemma~\ref{lem:Determines}, the Neumann cases where the $n_i$ are odd and at least $5$ give an infinite number of non-isomorphic finitely-generated groups in $\mathcal{G}$ with different sets of isomorphism types of finite simple normal subgroups.

\begin{proof}
First we consider the sequence $6 < 8 < 10 < \ldots$. We use the alphabet $\{{>}, {<}, 1, 2\}$ and the same rule as in the previous section to determine good runs. In good runs of length $n \leq 2$, we do nothing. For larger $n$, the automorphism corresponding to $a$ rotates the belt positions $(1; \; 2; \; 3)$ meaning at the beginning of each good run we perform the permutation $(1{<}{<}; \; {>}1{<}; \; {>}{>}1)$. The automorphism corresponding to $b$ is just the rotation of the head around the conveyor belt, as in the previous section. Obviously on finite good runs this is isomorphic to the natural action of the pseudo-Neumann group.

For other sequences, it suffices to construct the natural action of the group for a single sequence $\ell, \ell+k, \ell+2k, \ldots$, since we can then take the product action of Neumann groups corresponding to finitely many such sequences to get any Neumann group with an eventually periodic sequence.

Consider the alphabet $\{{>}, {<}, 1, 2\} \times \{1, 2, \ldots, k\}$. First, use the same rule as above to cut the sequence on the first track over $\{{>}, {<}, 1, 2\}$ into good runs. Then in each good run, on the second track over alphabet $\{1, 2, \ldots, k\}$, read maximal concatenations $w^n$ of the word $w = 1 2 \cdots k$. Cut the good word by these runs, including a full copy of $w$ on the left, and only the first $\ell$ letters in the rightmost good run. This forces good runs to be effectively of length $\ell \bmod k$. Now the generators with the same description as above (but with the new description of good runs) form the Neumann group.
\end{proof}

We note that the previous construction can also be performed in the topological full group $\llbracket \Sigma^\Z \rrbracket$ of $\Sigma^\Z$ for some alphabet $\Sigma$ (and therefore any nontrivial alphabet \cite{Sa21c}), by for a single arithmetic progression defining good runs to consist of maximal powers $w^n$ of the word $w = 12 \cdots (2k) \in \{1, \ldots, 2k\}^{2k}$ and dropping $2(n - \ell)$ positions on the right, and using the origin of the configuration as the head position (using odd and even positions instead of the symbol $s \in \{1, 2\}$ to determine whether we are on the top or bottom track of track of the conveyor belt). To get multiple arithmetic progressions, one can simply use finitely many disjoint alphabets.

%One obtains many interesting-looking groups from this construction. For example, restricting the movement of the belts to depend only on the relative position from the left end of the conveyor belt (acting by rotation by a constant when far enough from it), one obtains a residually finite cover of Houghton's group \cite{Le12a} which retains at least some of its basic properties, such as containing every finite group as a subgroup.

\section{Lemmas used in the introduction}

The following shows that after we restrict the base groups to be abelian, wreath products cannot escape co-NPness of the word problem.

\begin{lemma}
\label{lem:WreathNP}
Let $A, B$ be finitely generated groups, such that $A$ is abelian and $B$ has co-NP word problem. Then $A \wr B$ has co-NP word problem.
\end{lemma}

\begin{proof}
Let $\pi : A \wr B \to B$ be the natural projection. Pick a generating set $S$ for $B$. We exhibit polynomially checkable witnesses for the nontriviality of an element in the wreath product. Since $B$ has co-NP word problem, for elements with nontrivial $\pi(g)$ we can use the certificates of $B$. Elements with trivial $\pi(g)$ are finite products of the form
$\prod_i a_i^{b_i}$
where $a_i \in A, b_i \in S^*$. Writing $b_i \sim b_j$ for equality as elements of $B$, the nontriviality of the product means that for some $b$, the sum (in $A$) of elements $a_i$ corresponding to $b_i \sim b$ is nonzero. A polynomial witness then consists of witnesses for all non-equalities between unequal $b_i$ (using again the co-NPness of the word problem of $B$), or sufficiently many to conclude that the element is nontrivial.
\end{proof}

As noted in the footnote in the introduction, groups in $\mathcal{G}$ are in fact in co-NTIME$(n^d)$ for a fixed $d$, and the above proof does not show that wreath products could not increase the degree, at least without a more careful combinatorial analysis. It nevertheless seems unlikely that one could prove non-closure under wreath products this way, since no natural separating languages are known between co-NTIME$(n^d)$ and co-NTIME$(n^{d+1})$.

We cite \cite{MiVaWe19} as a general reference on decision problems about wreath products.

The following lemma shows that our Theorem~\ref{thm:MainUgly} does not provide examples of torsion-free groups in $\mathcal{G}$ which are not residually nilpotent (although it does not show that our main construction could not in principle provide examples). Note that $B \wr \Z^n$ gives torsion, so it suffices to consider the groups $C^n$.

\begin{lemma}
\label{lem:AllTFN}
If $A$ is a torsion-free abelian group, and $C$ is a free group, then $A \wr C^n$ is residually nilpotent.
\end{lemma}

\begin{proof}
In general, if $G, H$ are groups and $A$ is an abelian group, then a homomorphism $\pi : G \to H$ induces a homomorphism $\hat\pi : A \wr G \to A \wr H$, which maps as $\pi$ on $G$ and as identity on $A$.

Now, if $A$ is a torsion-free abelian group and $G$ is residually torsion-free nilpotent, then $A \wr G$ is residually nilpotent: For any nontrivial element $g \in A \wr G$, we can find a quotient $\pi : G \to H$ such that $\hat\pi : A \wr G \to A \wr H$ maps $g$ to a nontrivial element, and $H$ is torsion-free nilpotent. Since $A$ is torsion-free abelian and $H$ is torsion-free nilpotent, $A \wr H$ is residually nilpotent by \cite[Theorem~B2]{Ha70}, thus we can find a further quotient $\pi' : A \wr H \to N$ mapping $g$ to a nontrivial element in a nilpotent group. Since $g$ was arbitrary, $\pi' \circ \pi$ shows that $A \wr G$ is residually nilpotent.

The group $C^n$ is residually torsion-free nilpotent by \cite{Ma35}, so the claim is proved.
\end{proof}

% One might more generally ask whether wreath products can help with the problem of Kim and Roush in the first place, i.e.\ whether $A \wr H$ is in fact residually nilpotent whenever $H$ is residually nilpotent, torsion-free and finitely generated, and $A$ is finitely-generated torsion-free abelian. This is not the case [MO].

\bibliographystyle{plain}
\bibliography{../../../bib/bib}{}

\end{document}